\documentclass[a4paper,reqno,11pt]{amsart} 
\pdfoutput=1
\usepackage[T1]{fontenc}
\usepackage[utf8]{inputenc}

\usepackage{amsmath, amssymb, enumerate, color}
\usepackage[dvipsnames,svgnames,table]{xcolor}
\usepackage{todonotes}
\usepackage[inline]{enumitem}

\usepackage[labelsep = period, labelfont = bf, justification = centering]{caption}
\usepackage{float,graphicx,subcaption}
\DeclareCaptionSubType*[roman]{figure}

\newcommand{\defin}[1]{\emph{\textcolor{ForestGreen}{#1}}}

\usepackage{pgf,tikz,ifthen,calc}
\usepackage{tkz-euclide}
\tkzSetUpPoint[fill=black, size = 3pt]
\tkzSetUpLine[color=black, line width=0.6pt]
\tkzSetUpCircle[color=black, line width=0.6pt]

\usepackage{multirow}
\usepackage{longtable}
\usepackage{booktabs}
\usepackage{comment}

\makeatletter
\def\thm@space@setup{
  \thm@preskip=4mm
  \thm@postskip=0mm
}
\makeatother

\usepackage{hyperref}
\hypersetup{
  pdfauthor={Jędrzej Hodor, Freddie Illingworth, Tomasz Mazur},
  pdftitle={Treedepth and 2-treedepth in graphs with no long induced paths},
    colorlinks,
    linkcolor={RoyalBlue},
    citecolor={RubineRed},
    urlcolor={blue!80!black}
}
\usepackage[capitalise, noabbrev]{cleveref}
\usepackage{mathtools}
   
\theoremstyle{plain} %% default 

\newtheorem{theorem}{Theorem}
 
\newtheorem{problem}[theorem]{Problem}

\newtheorem{conjecture}[theorem]{Conjecture}
\newtheorem{lemma}[theorem]{Lemma}
\newtheorem{obs}[theorem]{Observation}

\newtheorem*{theorem*}{Theorem}
\newtheorem*{corollary*}{Corollary} 
\newtheorem*{problem*}{Problem}
\newtheorem*{conjecture*}{Conjecture}
\newtheorem*{lemma*}{Lemma}
\newtheorem*{obs*}{Observation}
\newtheorem*{remark*}{Remark}
\newtheorem*{proposition*}{Proposition}
\newtheorem*{question*}{Question} 
\newtheorem*{example*}{Example}

\theoremstyle{remark} 
\newtheorem{claim}[theorem]{Claim} 
\newtheorem*{claim*}{Claim} 
 
\newtheorem*{definition*}{Definition} 
\bibstyle{plain}

\usepackage{thmtools, thm-restate}

\newenvironment{proofclaim}[1][]
	{\par\noindent {\it Proof of the Claim}. }{ \hfill$\lozenge$\par\vspace{11pt}}

\DeclarePairedDelimiter\set{\{}{\}}

\DeclarePairedDelimiter\floor{\lfloor}{\rfloor}

\DeclarePairedDelimiter{\abs}{\lvert}{\rvert}

\newcommand{\Oh}{\mathcal{O}}

\newcommand{\calA}{\mathcal{A}} 
\newcommand{\calB}{\mathcal{B}} 
\newcommand{\calC}{\mathcal{C}}

\newcommand{\calO}{\mathcal{O}}
\newcommand{\calP}{\mathcal{P}} 
\newcommand{\calR}{\mathcal{R}}

\newcommand{\calT}{\mathcal{T}}

\newcommand{\calW}{\mathcal{W}}

\DeclareMathOperator\tw{tw}
\DeclareMathOperator\pw{pw}
\DeclareMathOperator\td{td}
\DeclareMathOperator\tdtwo{td_2}

\DeclareMathOperator\diam{diam}
\DeclareMathOperator\wcol{wcol}

\DeclareMathOperator\btree{F}

\renewcommand{\leq}{\leqslant}

\renewcommand{\geq}{\geqslant}

\let\subset\subseteq

\let\epsilon\varepsilon

\let\setminus\setminus

\graphicspath{{figs/}}

\renewenvironment{enumerate}{\begin{enumorig}[label=\textup{(\roman*)}, noitemsep, 
topsep=2pt plus 2pt, labelindent=.2em, leftmargin=*, widest=iii]}{\end{enumorig}}

\newenvironment{enumerateAlpha}{\begin{enumorig}[label=\textup{(\alph*)}, noitemsep, 
topsep=2pt plus 2pt, labelindent=.2em, leftmargin=*, widest=iii]}{\end{enumorig}}

\newenvironment{enumerateAlphaPrim}{\begin{enumorig}[label=\textup{(\alph*')}, noitemsep, 
topsep=2pt plus 2pt, labelindent=.2em, 
leftmargin=*, widest=iii]}{\end{enumorig}}

\usepackage[norule]{footmisc}


\begin{document} 
\title[Treedepth and 2-treedepth in graphs with no long induced paths] 
{ Treedepth and 2-treedepth in graphs with \linebreak no long induced paths}

\author[J.~Hodor]{J\k{e}drzej Hodor}
\author[F.~Illingworth]{Freddie Illingworth}
\author[T.~Mazur]{Tomasz Mazur}

\thanks{(J.~Hodor) \textsc{Theoretical Computer Science Department, 
Faculty of Mathematics and Computer Science and  Doctoral School of Exact and Natural Sciences, Jagiellonian University, Krak\'ow, Poland.} (F.~Illingworth) \textsc{Department of Mathematics, University College London, UK.} (T.~Mazur) \textsc{Theoretical Computer Science Department, 
Faculty of Mathematics and Computer Science, Jagiellonian University, Krak\'ow, Poland.}}

\thanks{\textit{E-mails:} \href{mailto:jedrzej.hodor@gmail.com}{jedrzej.hodor@gmail.com}, \href{mailto:f.illingworth@ucl.ac.uk}{f.illingworth@ucl.ac.uk}, \href{mailto:tom.mazur@student.uj.edu.pl}{tom.mazur@student.uj.edu.pl}}

\thanks{J.\ Hodor was supported by a Polish Ministry of Education and Science grant (Perły Nauki; PN/01/0265/2022).
F.\ Illingworth was supported by EPSRC grant EP/V521917/1 and the Heilbronn Institute for Mathematical Research.
T.\ Mazur was supported by the National Science Center of Poland under grant UMO-2023/05/Y/ST6/00079 within the Weave-UNISONO program.}

\thanks{\noindent\rule{\textwidth}{0.5pt}}

\begin{abstract}
    Huynh, Joret, Micek, Seweryn, and Wollan (Combinatorica, 2022) introduced a graph parameter, later referred to as 2-treedepth and denoted $\td_2(\cdot)$.
    The parameter is the natural 2-connected version of treedepth. 
    For every graph, 2-treedepth is at most the treedepth but can be much smaller: long paths have arbitrarily large treedepth but $2$-treedepth equal to $2$. 
    We prove a converse showing that every graph with no induced path on $t$ vertices and 2-treedepth at most $k$ has treedepth at most $g(k, t)$. In fact, we determine the value of the function $g$ up to a multiplicative factor of 2.
    
    %Known results show that for every positive integer $t$, there is a function $f$ such that every graph $G$ with no induced path on $t$ vertices has treedepth less than $f(\tdtwo(G))$.
    %We give asymptotically tight bounds for this function $f$.

    Additionally, we give asymptotically tight bounds for the problem of forcing long induced paths in graphs with long paths and bounded $2$-treedepth or bounded pathwidth.
    The latter result answers a question of Hilaire and Raymond (E-JC, 2023).
\end{abstract}

\maketitle

\thispagestyle{empty}

% =======================================================
\section{Introduction}\label{sec:introduction}
% =======================================================

\defin{Treedepth} is one of the key graph parameters in structural graph theory~\cite{sparsity}.
Recently, Huynh, Joret, Micek, Seweryn, and Wollan introduced its variant called \defin{$2$-treedepth}~\cite{HJMSW22}.
Both parameters admit recursive definitions.
We denote the treedepth (resp.\ $2$-treedepth) of a graph $G$ by \defin{$\td(G)$} (resp.\ \defin{$\tdtwo(G)$}), and we set
\begin{align*}
    \td(G)&=\begin{cases}
    0 & \textrm{if $G$ is the null graph,}\\
    \min_{v \in V(G)} \td(G-v) + 1 &\textrm{if $G$ is connected\footnotemark, and}\\
    \max_{i \in [k]}\td(C_i)&\textrm{if $G$ consists of components $C_1,\dots,C_k$ and $k > 1$;}
    \end{cases}\\
    \tdtwo(G)&=\begin{cases}
    0 & \textrm{if $G$ is the null graph,}\\
    \min_{v \in V(G)} \tdtwo(G-v) + 1 &\textrm{if $G$ is a block\footnotemark, and}\\
    \max_{i \in [k]}\tdtwo(B_i)&\textrm{if $G$ consists of blocks $B_1,\dots,B_k$ and $k > 1$.}
    \end{cases}
\end{align*}
\addtocounter{footnote}{-1}
\footnotetext{In this paper, connected graphs are nonnull, that is, they have at least one vertex. Note that a tree is defined as a connected forest, thus, trees and subtrees are also assumed to be nonnull.}
\addtocounter{footnote}{+1}
\footnotetext{A \defin{cut-vertex} of a graph $G$ is a vertex $v\in V(G)$ such that $G-v$ has more components than $G$.
A \defin{block} of $G$ is a maximal connected subgraph of $G$ without a cut-vertex.
Note that the blocks can be of three types: maximal $2$-connected subgraphs, cut edges together with
their endpoints, and isolated vertices.
Two blocks have at most one vertex in common, and such a vertex is always a cut-vertex.}

The new parameter gained some attention~\cite{block-elimination-distance,games,k-td} and was found to be useful in a rather unexpected context~\cite{wcols,centered}.
One of the critical properties of $2$-treedepth is describing classes of graphs excluding a fixed ladder as a minor (for a positive integer $t$, the ladder $L_t$ is the $2\times t$ grid).
More precisely, a class of graphs closed under taking minors excludes some ladder if and only if there exists a constant bounding the $2$-treedepth of all the graphs in the class~\cite{HJMSW22}.
This is analogous to the folklore statement that excluding a path is equivalent to bounding the treedepth.

By their definitions, every graph $G$ satisfies $\tdtwo(G) \leq \td(G)$.
Furthermore, paths have unbounded treedepth, while the $2$-treedepth of any path is at most $2$.
A natural question arises: for which classes of graphs does there exist a function\footnote{All functions that we consider in this paper map positive integers to positive integers.} $g$ such that every graph $G$ in the class satisfies $\td(G) \leq g(\tdtwo(G))$?
The answer is clear for classes of graphs closed under taking minors or under taking subgraphs.
As long as such a class excludes a path, both parameters are bounded by an absolute constant, so $g$ exists trivially.
When such a class contains all paths, $g$ cannot exist.

To complete the picture, we consider \defin{hereditary} classes of graphs (that is, classes of graphs closed under taking induced subgraphs).
Again, if such a class contains all paths, then $g$ does not exist.
Now define \defin{$g(k, t)$} to be the greatest treedepth amongst graphs that are $P_t$-free\footnote{For a graph $H$, we say that a graph $G$ is \defin{$H$-free} when $G$ does not contain $H$ as an induced subgraph.
For a positive integer $t$, we denote the path on $t$ vertices (in other words, the \defin{path of order $t$}) by \defin{$P_t$}.} and have 2-treedepth at most $k$. It is possible to use known results to show that $g(k, t)$ is always finite. 
For example, statements on long paths in sparse graphs forcing long induced paths~\cite{long-induced-paths-sparse} show that $g(k, t)$ is finite while bounds on weak colouring numbers in classes of graphs of bounded treewidth~\cite{Grohe15} show that $g(k, t) \leq \binom{t + k}{t - 1} = \calO_t(k^{t - 1})$.
We discuss these in more detail in \cref{sec:known}.

We significantly improve the bound on $g$ (showing that $g(k, t) = \calO_t(k^{\floor{(t - 1)/2}}$) and, in fact, determine $g(k, t)$ up to a factor of two.

\begin{theorem} \label{thm:main}
    For every integer $t \geq 2$, for every integer $k \geq 1$, and for every $P_t$-free graph $G$ with $\tdtwo(G) \leq k$,
    \[
    \td(G) < 2 \binom{\lfloor (t-1) \slash 2 \rfloor+k-1}{\lfloor (t-1) \slash 2 \rfloor}.
    \]
\end{theorem}

A slight modification (in fact, simplification) of a construction of Grohe, Kreutzer, Rabinovich, Siebertz, and Stavropoulos~\cite{Grohe15} shows that \cref{thm:main} is within a factor of two of best possible.
More precisely, we have the following.

\begin{theorem}\label{thm:lower-bound}
    For all integers $t \geq 2$ and $k \geq 1$, there exists a $P_t$-free graph $G$ with $\tdtwo(G) \leq k$ such that
        \[\td(G) \geq \binom{\lfloor (t-1) \slash 2 \rfloor+k-1}{\lfloor (t-1) \slash 2 \rfloor}.\]
\end{theorem}

$P_2$-free graphs have no edges and $P_3$-free graphs are disjoint unions of cliques and so $g(k, 3) = k$.
We also determine the function $g$ exactly for $t \in \set{4, 5}$ (note that, for $t = 5$, the lower bound from \cref{thm:lower-bound} is $\binom{k + 1}{2}$).

\vbox{
\begin{theorem} \label{thm:p5free-upper}
    Let $G$ be a graph and let $\tdtwo(G) = k$.
    \begin{enumerate}
        \item If $G$ is $P_4$-free, then $\td(G) = k$, \label{item:p4}
        \item If $G$ is $P_5$-free, then $\td(G) \leq \binom{k+1}{2}$. \label{item:p5}
    \end{enumerate}
\end{theorem}
}

\subsection{Forcing long induced paths}\label{sec:forcing}
Suppose that a graph $G$ contains a path of order $n$ as a subgraph.
What can we say about the longest induced path in $G$?
In general, there can be no non-trivial induced path, e.g.\ when $G$ is a complete graph.
In 1982, Galvin, Rival, and Sands~\cite{Galvin82} proved that for a fixed positive integer $t$, there exists a function $f$ (that tends to infinity with its input) such that, for every positive integer $n$ and for every graph $G$ that excludes $K_{t,t}$ as a subgraph, if $G$ contains $P_n$ as a subgraph, then $G$ contains $P_{f(n)}$ as an induced subgraph.
One can study the growth of such a function assuming that $G$ is in a more restrictive graph class; see a summary of the known results~\cite[Figure~1]{long-induced-paths-sparse} as well as the recent work of Hunter, Milojević, Sudakov, and Tomon~\cite{longinducedpaths-nokss}.
For a graph class $\calC$, let \defin{$f(\calC,\cdot)$} be a function taking the largest possible values such that for every positive integer $n$ and for every $G \in \calC$, if $G$ contains $P_n$ as a subgraph, then $G$ contains $P_{f(\calC,n)}$ as an induced subgraph.

Hilaire and Raymond~\cite{Hilaire2023} proved that for every positive integer $k$, if $\calP_k$ is the class of graphs with pathwidth\footnote{We denote the pathwidth of a graph $G$ as \defin{$\pw(G)$} and we give the formal definition in \cref{sec:forcing:proofs}.} at most $k$, then\footnote{The statement of \cite[Theorem~1.5]{Hilaire2023} is $f(\calP_{k - 1}, n) \geq 1/3 \cdot n^{1/k}$. However, a careful reading of the argument shows that the same proof gives the improved bound $f(\calP_k, n) \geq 1/3 \cdot n^{1/k}$. \label{foot:improved}} $f(\calP_{k},n)$ is in $\Oh(n^{2 \slash (k + 1)})$ and $\Omega(n^{1 \slash k})$.
We show that the latter bound is correct and so $f(\calP_{k},n) = \Theta(n^{1 \slash k})$. 

\vbox{
\begin{theorem}\label{thm:pw}
    For every positive integer $k$, there exists a constant $c_k > 0$ such that for every positive integer $n$, there exists a graph $G$ with $\pw(G) \leq k$ that contains $P_n$ as a subgraph but the longest induced path in $G$ is of order at most $c_k \cdot n^{1 \slash k}$.
    %For all integers $\ell \geq 2$ and $k \geq 1$, there exists a graph $G_{\ell,k}$ with pathwidth $k$ that contains $P_{\ell^k/k!}$ as a subgraph but not $P_{\ell + 1}$ as an induced subgraph.
\end{theorem}
}

Furthermore, we give tight bounds for the class of graphs of bounded $2$-treedepth.
That is, for every integer $k \geq 2$, if $\calT_k$ is the class of graphs with $2$-treedepth at most $k$, then we show that $f(\calT_k,n) = \Theta(n^{1/(k - 1)})$.
In fact, for the lower bound, we use the same construction as in \cref{thm:pw}.

\begin{restatable}{theorem}{tdupper}\label{thm:td2-upper}
    For all integers $k \geq 2$ and $n \geq 1$, for every graph $G$ with $\tdtwo(G) \leq k$, if $G$ contains $P_n$ as a subgraph, then $G$ contains a path of order at least $n^{1/(k-1)}/2$ as an induced subgraph.
\end{restatable}

\begin{theorem}\label{thm:td2-lower}
    For every integer $k\geq 2$, there exists a constant $c_k > 0$ such that for every positive integer $n$, there exists a graph $G$ with $\tdtwo(G) \leq k$ that contains $P_n$ as a subgraph but the longest induced path in $G$ is of order at most $c_k \cdot n^{1 \slash (k-1)}$.
    %For all integers $\ell \geq 2$ and $k \geq 1$, there exists a graph $G_{\ell, k}$ with $2$-treedepth $k + 1$ that contains $P_{\ell^k/k!}$ as a subgraph but not $P_{\ell + 1}$ as an induced subgraph.
\end{theorem}

\subsection{Outline of the paper}

In \cref{sec:known}, we study the weak variants of \cref{thm:main}.
In \cref{sec:preliminaries} we settle notation and we introduce basic tools.
In \cref{sec:proof}, we prove \cref{thm:main}.
In \cref{sec:short-paths}, we prove \cref{thm:p5free-upper}.
In \cref{sec:lower-bound}, we prove \cref{thm:lower-bound}.
In \cref{sec:forcing:proofs}, we prove \cref{thm:pw,thm:td2-lower,thm:td2-upper}.
In \cref{sec:open}, we state some open problems.

\section{Weak versions of \texorpdfstring{\cref{thm:main}}{main theorem}}\label{sec:known}

In this section, we show two ways of obtaining a weak version of \cref{thm:main}.

\subsection{First weak version of \texorpdfstring{\cref{thm:main}}{main theorem}}

Duron, Esperet, and Raymond~\cite[Corollary 1.4]{long-induced-paths-sparse} proved that if a minor-closed class of graphs $\calC$ excludes an outerplanar graph, then $f(\calC,\cdot)$ grows polynomially.
It is easy to check that for every positive integer $t$, the ladder $L_t$ satisfies $\tdtwo(L_t) \geq \log t$.
Therefore, for every positive integer $k$, the class of graphs with $2$-treedepth at most $k$ excludes $L_{2^{k+1}}$.
From~\cite[Theorem 4.8]{long-induced-paths-sparse} and the proof of~\cite[Corollary 1.4]{long-induced-paths-sparse} one can deduce the following statement.

\vbox{
\begin{lemma}[\cite{long-induced-paths-sparse}]\label{lem:DER}
    For every positive integer $\ell$ there exists a constant $c_\ell > 0$ such that for every positive integer $n$, for every graph $G$ with no $L_\ell$ as a minor, if $G$ contains $P_n$ as a subgraph, then $G$ contains a path of order at least $c_\ell \cdot n^{1 \slash \ell}$ as an induced subgraph.
\end{lemma}
}

Let $t$ and $k$ be positive integers.
Let $G$ be a $P_t$-free graph with $\tdtwo(G) \leq k$.
As discussed, $G$ has no $L_{2^{k+1}}$ as a minor.
Let $c_{2^{k+1}} \leq 1$ be a constant from \cref{lem:DER} and let $n$ be the smallest positive integer such that $c_{2^{k+1}} \cdot n^{1 \slash 2^{k+1}} \geq t$.
In particular, such $n$ depends only on $k$ and $t$, however, $n = 2^{\Omega_t(2^k)}$.
\Cref{lem:DER} implies that $G$ has no $P_n$ as a subgraph.
Therefore, $\td(G) < n$.

\subsection{Second weak version of \texorpdfstring{\cref{thm:main}}{main theorem}}

Here, we will use several notions without defining them properly, as we will use them only as black boxes.
The first such notion is the notion of weak colouring number.
This is a family of graph parameters parametrized by positive integers and $\infty$.
For such $r$, we denote by $\wcol_r(G)$ the $r$th weak colouring number of a graph~$G$.

The idea that we exploit was already used by Bonamy,  Bousquet, Pilipczuk, Rzążewski, Thomassé, and Walczak~\cite[Lemma~29]{BONAMY2022353}.
The following lemma gathers two straightforward facts on weak colouring numbers. 
The first one follows directly from the definition and for the proof of the second, see e.g.\ lecture notes by Pilipczuk, Pilipczuk, and Siebertz~\cite[Chapter~1, Theorem~1.17]{notes}.

\begin{lemma}[Folklore]\label{wcol}
    Let $G$ be a graph and let $t \geq 3$ be an integer.
    \begin{enumerate}
        \item If $G$ is $P_t$-free, then $\wcol_{t-1}(G) = \wcol_\infty(G)$.\label{wcol:P_t-free}
        \item $\wcol_\infty(G) = \td(G)$.\label{wcol:wcol_infty}
    \end{enumerate}
\end{lemma}

The treewidth of a graph $G$ is denoted by $\tw(G)$.
Grohe et al.~\cite{Grohe15} showed a bound on weak colouring numbers in the class of graphs of bounded treewidth.

\begin{theorem}\label{wcol:bd:tw}
    Let $G$ be a graph, and let $k$ and $r$ be positive integers.
    If $\tw(G) \leq k$, then $\wcol_r(G) \leq \binom{k+r}{r}$.
\end{theorem}

The final piece that we need is that for every graph $G$, we have $\tw(G)+1 \leq \tdtwo(G)$.
This follows e.g.\ from the equivalent definition of treewidth via clique sums.
For more details, see the work of Rambaud~\cite{k-td}.

Finally, let $G$ be a $P_t$-free graph with $\tdtwo(G) \leq k$.
In particular, $\tw(G) \leq k-1$.
Respectively, by \cref{wcol}.\ref{wcol:wcol_infty}, \cref{wcol}.\ref{wcol:P_t-free}, and \cref{wcol:bd:tw}, we have
    \[\td(G) = \wcol_\infty(G) = \wcol_{t-1}(G) \leq \binom{k+t-2}{t-1} = \mathcal{O}(k^{t-1}).\]

\section{Preliminaries}\label{sec:preliminaries}
% =======================================================

Let $G$ be a graph.
A vertex $v$ of $G$ is an \defin{apex} in $G$ if $v$ is adjacent in $G$ to every $u \in V(G - v)$.

The \defin{length} of a path is its number of edges.
Note that for every positive integer $t$, the length of $P_t$ is $t-1$.
The \defin{distance} between two vertices $u$ and $v$ in the same component of $G$ is the minimum length of a path between $u$ and $v$ in $G$.
Every path between $u$ and $v$ of length equal to the distance between $u$ and $v$ is called a \defin{shortest path} in $G$.
When $G$ is connected, then the \defin{diameter} of $G$ is the maximum distance between any two vertices in $G$.
We denote the diameter of $G$ by \defin{$\diam(G)$}.

Given a path $P = v_1\cdots v_n$ in $G$, we say that there is a \defin{shortcut} between $v_i$ and $v_j$ in $P$ if $|i - j| > 1$ and $v_iv_j$ is an edge in $G$.
We say that $P$ is an \defin{induced path} in $G$, if there is no shortcut in $P$.
Note that if $P$ is a shortest path in $G$, then it is also an induced path in $G$.
Given paths $P = v_1 \cdots v_n v$ and $Q = vu_1 \cdots u_m$, such that their only common vertex is $v$, we define the concatenation of $P$ and $Q$ as $\defin{PvQ} = v_1 \cdots v_n v u_1 \cdots u_m$.
If $G$ is a graph and $P = v_1 \cdots v_n$ and $Q = u_1 \cdots u_m$ are disjoint paths in $G$ such that $v_nu_1$ is an edge in $G$, then we define the concatenation of $P$ and $Q$ as $\defin{PQ} = v_1 \cdots v_n u_1 \cdots u_m$.

We denote by \defin{$\calC(G)$} the set of components of $G$.
A vertex $v$ of $G$ is a \defin{cut-vertex} of $G$ if $|\calC(G - v)| > |\calC(G)|$. 
We write \defin{$\calA(G)$} for the set of all cut-vertices of $G$. 
A subgraph $B \subseteq G$ is a \defin{block} of $G$ if $B$ is a maximal subgraph of $G$ such that $B$ has no cut-vertex. 
We use \defin{$\calB(G)$} to denote the set of blocks of $G$. 
The \defin{forest of blocks} of $G$ is the graph \defin{$\btree(G)$} with vertex set $\calB(G) \cup \calA(G)$ such that we place an edge between a block $B$ and a cut-vertex $v$ whenever $v \in V(B)$. 
It is easy to see that $\btree(G)$ is a forest.
See \cref{fig:btree} for an example.
The next observation summarizes some basic properties of forests of blocks.

\begin{obs}
    Let $G$ be a graph.
    \begin{enumerate}
        \item \label{obs:connected} $G$ is connected if and only if $\btree(G)$ is a tree.
        \item \label{obs:bipartite} $\btree(G)$ is bipartite with every edge being between vertices in $\calA(G)$ and $\calB(G)$.
        \item \label{obs:leafs} All leaves of $\btree(G)$ are in $\calB(G)$.
        \item \label{obs:odd} If $G$ is nonnull, then every maximal path in $F(G)$ has odd number of vertices.
        \item \label{obs:even-diam} If $G$ is connected, then $\diam(\btree(G))$ is even.
        %\item \label{obs:monotonicity} For every cut-vertex $v$, there exists a graph $G_v \supset G - v$ such that $\btree(G_v) = \btree(G) - v$. 
    \end{enumerate}
\end{obs}

\begin{lemma}\label{lem:apex-td2}
    Let $G$ be a connected graph and let $v \in V(G)$ be an apex in $G$.
    Then, $\tdtwo(G) = \tdtwo(G - v) + 1$ and $\td(G) = \td(G - v) + 1$.
\end{lemma}

\begin{proof}
    It is clear that $\td(G) \leq \td(G-v)+1$.
    Next, let $u \in V(G)$ be such that $\td(G) = \td(G-u)+1$.
    Note that $G-v$ is isomorphic to a subgraph of $G - u$, hence, $\td(G-u) \geq \td(G-v)$.
    It follows that $\td(G) = \td(G-u)+1 \geq \td(G-v)+1$, which ends the proof for treedepth.

    Let $C$ be a component of $G - v$ and let $H = G[V(C) \cup \{v\}]$.
    We claim that $\td_2(H) = \td_2(C) + 1$.
    It is clear that $\td_2(H) \leq \td_2(C)+1$.
    Note that $H$ is a block of $G$.
    Let $u \in V(H)$ be such that $\td_2(H) = \td_2(H-u)+1$.
    Note that $H-v$ is isomorphic to a subgraph of $H - u$, hence, $\td_2(H-u) \geq \td_2(H-v)$.
    It follows that $\td_2(H) = \td_2(C) + 1$, as claimed.
    We have
    \[\td_2(G) = \max_{H \in \calB(G)} \td_2(H) = \max_{C \in \calC(G-v)} \td_2(C) + 1 = \td_2(G-v) + 1.\]
    This completes the proof.
    % Let $\mathrm{p} \in \{\td,\tdtwo\}$.
    % %It is clear that $\mathrm{p}(G) \leq \mathrm{p}(G-v)+1$.
    % If $\mathrm{p} = \td$, let $H = G$ and if $\mathrm{p} = \td_2$, let $H$ be a block of $G$ such that $\td_2(G) = \td_2(H)$.
    % Next, let $u \in V(H)$ be such that $\mathrm{p}(G) = \mathrm{p}(H) = \mathrm{p}(H-u)+1$.
    % Note that $H-v$ is isomorphic to a subgraph of $H - u$, hence, $\mathrm{p}(H-u) \geq \mathrm{p}(H-v)$.
    % It follows that $\mathrm{p}(G) = \mathrm{p}(H-u)+1 \geq \mathrm{p}(H-v)+1 \geq \mathrm{p}(H) = \mathrm{p}(G)$, which ends the proof.
    % We prove the statement for 2-treedepth, as the proof for treedepth is identical. Let $u \in V(G)$ be taken such that $\tdtwo(G) = \tdtwo(G - u) + 1$. By definition of $v$, $N_G(u) \cup \{u\} \subseteq N_G(v) \cup \{v\} = V(G)$. Therefore, $G - v$ is isomorphic to the subgraph of $G - u$ achieved by setting $N_{G-u}(v) = N_{G-v}(u)$ and thus $\tdtwo(G - v) \leqslant \tdtwo(G - u)$, so $\tdtwo(G) \geqslant \tdtwo(G - v) + 1$. By definition, $\tdtwo(G) \leqslant \tdtwo(G - v) + 1$, so indeed $\tdtwo(G) = \tdtwo(G - v) + 1$.
\end{proof}

\begin{figure}[tp]
    \centering
    \includegraphics{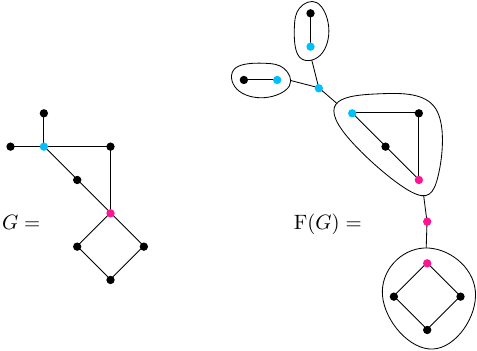}
    \caption{A graph $G$ and its corresponding $\btree(G)$.}
    \label{fig:btree}
\end{figure}

\section{Treedepth of \texorpdfstring{$P_t$}{Pt}-free graphs}\label{sec:proof}

In this section, we prove \cref{thm:main}.
In fact, we prove a slightly stronger result, \cref{lem:td-f}.
To state it, we need some more notation.

Given a graph $G$ and $S \subset V(G)$, we say that a path in $G$ is an \defin{$S$-path} if both its endpoints are in $S$.
Following Hodor, La, Micek, and Rambaud~\cite{quickly-excluding-apex-forest}, we define a variant of treedepth focused on a prescribed set of vertices.
The definition is recursive.
For a graph $G$ and $S \subset V(G)$, we denote their treedepth by \defin{$\td(G,S)$}, and we set
\[
   \td(G,S) =
   \begin{cases}
     0 & \text{if $S=\emptyset$,} \\
     \max\{\td(C, S \cap V(C)) : C\in\mathcal{C}(G)\} & \text{if $S\neq\emptyset$ and $|\mathcal{C}(G)|>1$,} \\
     1+\min\{\td(G-v, S \setminus \{v\}) : v\in V(G)\} & \text{otherwise.}
   \end{cases}
\]
Let us mention several key properties of this notion -- let $G$ be a graph and $S \subset V(G)$.
First, we have $\td(G) = \td(G,V(G))$.
Second, we have $\td(G,S) \leq |S|$.
Finally, we have the following simple statement.

\begin{lemma}\label{lem:apply-elimination}
    Let $G$ be a graph and let $S,S' \subset V(G)$.
    There exists $X \subset V(G)$ such that $S' \subset X$ and 
        \[\td(G,S) \leq \td(G,S') + \td(G - X, S \setminus X).\]
\end{lemma}
\begin{proof}
    The proof is by induction on $\ell := \td(G,S')$.
    If $\ell = 0$, then $X = \emptyset$ satisfies the required properties.
    Thus, assume that $\ell > 0$.
    For each component $C \in \calC(G)$, let $X_C$ be empty if $C$ has no vertices of $S'$, and let $X_C$ be a single vertex $v_C$ such that $\td(C,S' \cap V(C)) = \td(C-v_C,(S'\cap V(C)) \setminus \{v_C\}) + 1$ otherwise.
    Next, let $X' = \bigcup_{C \in \calC(G)} X_C$.
    It follows that $\td(G,S') = \td(G - X',S' \setminus X') + 1$.
    Moreover, $\td(G,S) \leq \td(G - X', S \setminus X') + 1$ as $\td(C,S\cap V(C)) \leq \td(C - v_C,(S \cap V(C)) \setminus \{v_C\}) + 1$ for every $C \in \calC(G)$.
    By induction applied to $G - X'$, $S \setminus X'$, and $S' \setminus X'$, there exists a set $X''$ with $S' \setminus X' \subset X'' \subset V(G -X')$ such that 
    \[ \td(G-X',S\setminus X') \leq \td(G-X',S' \setminus X') + \td(G - (X' \cup X''), S \setminus (X' \cup X'')).\]
    We set $X = X' \cup X''$ and altogether we obtain
    \begin{align*}
        \td(G,S) &\leq \td(G - X', S \setminus X') + 1\\
                 &\leq \td(G-X',S' \setminus X') + 1 + \td(G - X, S \setminus X)\\
                 &= \td(G,S') + \td(G - X, S \setminus X).\qedhere
    \end{align*}
\end{proof}

Let $G$ be a graph.
We say that an induced subgraph $H$ of $G$ is \defin{component-wise connected in $G$} if for every $C \in \calC(G)$, either $V(H) \cap V(C) = \emptyset$ or $G[V(H) \cap V(C)]$ is connected.
Let $S \subset V(G)$.
When $\calB' \subset \calB(G)$, we say that $G' = G[\bigcup_{B \in \calB'} V(B)]$ is an \defin{$S$-core} of $G$ \defin{induced} by $\calB'$ if $G'$ is component-wise connected in $G$ and $S \subset V(G')$.
We say that an $S$-core $G'$ of $G$ induced by $\calB'$ is \defin{minimal} if $\calB'$ is an inclusion-wise minimal subset of $\calB(G)$ inducing an $S$-core.
A crucial property of $S$-cores is encapsulated in the following quite straightforward statement.

\begin{lemma}\label{lem:S-cores}
    Let $G$ be a graph, let $S \subset V(G)$, let $G'$ be an $S$-core of $G$, let $C \in \calC(G)$ be such that $S \cap V(C) \neq \emptyset$, and let $v \in V(G')$.
    \begin{enumerate}
        \item $G[V(G') \cap V(C)]$ is an $(S\cap V(C))$-core of $C$.\label{lem:S-cores:components}
        \item $G'-v$ is an $(S \setminus \{v\})$-core of $G-v$.\label{lem:S-cores:vertices}
    \end{enumerate}
\end{lemma}
\begin{proof}
    Suppose that $G'$ is induced by $\calB' \subset \calB(G)$.
    Let $\calB_C = \{B \in \calB' : V(B) \cap V(C) \neq \emptyset\}$.
    Note that $\calB_C \subset \calB(C)$ and $S \cap V(C) \subset V(G') \cap V(C)$.
    Since $G'$ is component-wise connected in $G$, $G[V(G') \cap V(C)]$ is connected, and so, component-wise connected in $C$.
    Finally, $V(G') \cap V(C) = \bigcup_{B \in \calB_C} V(B)$.
    Altogether, we obtain that $G[V(G') \cap V(C)]$ is an $(S\cap V(C))$-core of $C$, and so,~\ref{lem:S-cores:components} holds.
    
    Since $S \subset V(G')$, we have $S \setminus \{v\} \subset V(G'-v)$.
    For every block $B \in \calB(G-v)$ there exists a block $\beta(B) \in \calB(G)$ such that $V(B) \subset V(\beta(B))$.
    Note that $V(B) \neq V(\beta(B))$ if and only if $v \in V(\beta(B))$.
    Let $\calB_v = \{B \in \calB(G-v) : \beta(B) \in \calB'\}$.
    In particular, we have $V(G'-v) = \bigcup_{B \in \calB_v}V(B)$.
    It suffices to show that $G'-v$ is component-wise connected.
    % If $C \in \calC(G-v)$ and $V(C) = V(C')$ for some $C' \in \calC(G)$, then $V(G'-v) \cap V(C) = V(G) \cap V(C')$ and so $G[V(G'-v) \cap V(C)]$ is connected.
    Suppose to the contrary that there is $D \in \calC(G-v)$ such that $G[V(G'-v) \cap V(D)]$ is not connected.
    Then, there exists $C' \in \calC(G)$ such that $V(D) \subset V(C')$ and $v \in V(C')$.
    Since $G'$ is component-wise connected in $G$, the graph $G[V(G') \cap V(C')]$ is connected.
    Let $D_1,D_2$ be distinct components of $G[V(G'-v) \cap V(D)]$.
    It follows that $V(D_1)$ and $V(D_2)$ are adjacent to $v$ in $G$.
    Next, let $B_1,B_2 \in \calB_v$ be blocks contained in $D_1$ and $D_2$ respectively such that $V(B_1)$ and $V(B_2)$ are adjacent to $v$ in $G$. 
    Let $H$ be a cycle in $G$ obtained by adding $v$ to a path in $D$ between vertices of $B_1$ and $B_2$ adjacent to $v$.
    In particular, all vertices of $H$ are in the same block of $G$.
    By the definition of $\calB_v$, it follows that either all vertices of $H$ are in $G'-v$ or none of them is.
    This is a contradiction that concludes the proof of~\ref{lem:S-cores:vertices}.
\end{proof}

\begin{lemma}\label{lem:S-core}
    Let $G$ be a graph, let $S \subset V(G)$, and let $G'$ be an $S$-core of $G$.
    Then,
        \[\td(G,S) = \td(G',S).\]
\end{lemma}
\begin{proof}
    The proof is by induction on $|V(G)|$.
    When $|V(G)| = 0$ or even $|S| = 0$, the assertion clearly holds.
    Thus, we assume $|S| > 0$.
    We have $\td(G',S) \leq \td(G,S)$, hence, it suffices to show $\td(G,S) \leq \td(G',S)$.
    
    First, assume that $|\calC(G)| > 1$.
    Let $C \in \calC(G)$ be such that $\td(G,S) = \td(C,S \cap V(C))$.
    By \cref{lem:S-cores}.\ref{lem:S-cores:components}, $C' = G[V(G') \cap V(C)]$ is an $(S\cap V(C))$-core of $C$.
    Thus, applying induction to $C$, $S \cap V(C)$, and $C'$, we obtain $\td(C,S \cap V(C)) = \td(C',S \cap V(C))$.
    Again, by definition, we also have $\td(C',S \cap V(C)) \leq \td(G',S)$.
    Altogether, as desired:
    \[\td(G,S) = \td(C,S\cap V(C)) = \td(C',S \cap V(C)) \leq \td(G',S).\]

    Next, assume that $G$ is connected.
    Let $v \in V(G')$ be such that $\td(G',S) = \td(G'-v,S\setminus \{v\}) + 1$. 
    Note that $\td(G,S) \leq \td(G-v,S\setminus \{v\}) + 1$.
    By \cref{lem:S-cores}.\ref{lem:S-cores:vertices}, $G'-v$ is an $(S \setminus \{v\})$-core of $G-v$.
    Therefore, applying induction to $G-v$, $S\setminus \{v\}$, and $G'-v$, we obtain $\td(G-v,S\setminus \{v\}) = \td(G'-v,S\setminus\{v\})$.
    Altogether, as desired:
    \[\td(G,S) \leq \td(G-v,S\setminus \{v\}) + 1 = \td(G'-v,S\setminus\{v\}) + 1 =  \td(G',S). \qedhere\]
\end{proof}

We define recursively a function $f(k,t)$ for all integers $k \geq 1$ and $t \geq 2$:
    \[
       f(k,t) =
       \begin{cases}
         1 & \text{if $k = 1$,} \\
         1 & \text{if $t = 2$,} \\
         k & \text{if $t = 3$,} \\
         f(k,t-2) + f(k-1,t) + 1 & \text{otherwise.}
       \end{cases}
    \]
Note that $f$ is nondecreasing, we will use this fact implicitly.
\begin{lemma}\label{lem:computation}
    For all positive integers $k \geq 1$ and $t \geq 2$, we have
        \[f(k,t) < 2 \binom{\lfloor (t-1) \slash 2 \rfloor+k-1}{\lfloor (t-1) \slash 2 \rfloor}.\]
\end{lemma}

\begin{proof}
    We will in fact prove that, for all even $t$,
    \begin{equation}\label{eq:f}
        f(k, t) = 2 \binom{\lfloor (t-1) \slash 2 \rfloor+k-1}{\lfloor (t-1) \slash 2 \rfloor} - 1.
    \end{equation}
    The result will then follow, since for odd $t$,
    \[
        f(k, t) \leq f(k, t + 1) = 2 \binom{\lfloor t \slash 2 \rfloor+k-1}{\lfloor t \slash 2 \rfloor} - 1 = 2 \binom{\lfloor (t-1) \slash 2 \rfloor+k-1}{\lfloor (t-1) \slash 2 \rfloor} - 1.
    \]
    We will prove \eqref{eq:f} by induction on $k + t$. For $t = 2$ or $k = 1$, \eqref{eq:f} holds as both sides are equal to 1. Fix $k > 1$ and an even $t > 2$. By the recursion and induction, we have
    \begin{align*}
        f(k, t) & = f(k, t - 2) + f(k - 1, t) + 1 \\
        & = 2 \binom{\lfloor (t-3) \slash 2 \rfloor+k-1}{\lfloor (t-3) \slash 2 \rfloor} - 1 + 2 \binom{\lfloor (t-1) \slash 2 \rfloor+k-2}{\lfloor (t-1) \slash 2 \rfloor} - 1 + 1 \\
        & = 2\biggl[\binom{\lfloor (t-1) \slash 2 \rfloor+k-2}{\lfloor (t-1) \slash 2\rfloor - 1} + \binom{\lfloor (t-1) \slash 2 \rfloor+k-2}{\lfloor (t-1) \slash 2 \rfloor} \biggr] - 1 \\
        & = 2 \binom{\lfloor (t-1) \slash 2 \rfloor+k-1}{\lfloor (t-1) \slash 2 \rfloor} - 1. \qedhere
    \end{align*}
\end{proof}

The next lemma implies \cref{thm:main}, by taking $S = V(G)$ and applying \cref{lem:computation}.

\begin{lemma}\label{lem:td-f}
    For all integers $k \geq 1$ and $t \geq 2$, for every graph $G$ and every $S\subset V(G)$ with no induced $S$-path on at least $t$ vertices such that $\tdtwo(G) \leq k$, we have
        \[\td(G,S) \leq f(k,t).\]
\end{lemma}
\begin{proof}
    The proof is by induction on $k+t$.
    Let $G$ be a graph and $S\subset V(G)$.
    If $S = \emptyset$, then the statement is immediate, so assume $S \neq \emptyset$.
    Assume that $G$ has no induced $S$-path on at least $t$ vertices and that $\tdtwo(G) \leq k$.
    Without loss of generality, we can assume that $G$ is connected.
    When $k = 1$, then $G$ has no edges and clearly $\td(G,S) \leq 1$.
    When $t = 2$, then $G$ has no $S$-path, and so, each component of $G$ has at most one vertex in $S$. 
    It follows that $\td(G,S) \leq 1$.
    When $t = 3$, then $G$ has no induced $S$-path on at least $3$ vertices.
    In this case, each two vertices in $S$ are connected by an edge, as otherwise, since $G$ is connected, there is an induced $S$-path of order at least $3$ in~$G$.
    Therefore, $G[S]$ is a clique. 
    In particular, $|S| \leq \tdtwo(G) \leq k$, and so, $\td(G,S) \leq |S| \leq k$.

    Suppose that $k > 1$ and $t > 3$.
    %For convenience, we assume that $|S| \geq 2$.
    Let $G'$ be an inclusion-wise minimal $S$-core in $G$.
    By \cref{lem:S-core}, we have $\td(G,S) = \td(G',S)$.
    Let $S'$ be the set of all cut-vertices of $G'$, i.e.\ $S' = \calA(G')$.
    We claim that $G'$ has no induced $S'$-path on at least $t-2$ vertices.
    Suppose to the contrary that $G'$ has such a path $P$ and take it inclusion-wise maximal.
    Let $u_1,u_2 \in S'$ be the endpoints of $P$.
    Let $Q$ be the unique path between $u_1$ and $u_2$ in $\btree(G')$.
    It follows that all the vertices of $P$ lie in $\bigcup_{B \in \calB(G') \cap V(Q)} V(B)$.
    Let $Q'$ be a maximal path in $\btree(G')$, which contains $Q$ as a subpath.
    The path $Q'$ contains two leaves $B_1$ and $B_2$ of $\btree(G')$ (assume that $u_i \in V(B_i)$ for each $i \in [2]$).
    Note that for each $i \in [2]$, we have $S \cap V(B_i -u_i) \neq \emptyset$ as otherwise, $\calB(G') \setminus \{B_i\}$ induced a smaller $S$-core in $G$, which contradicts the minimality of $G'$.
    Using a vertex in $S \cap V(B_1)$ and $S \cap V(B_2)$, we can extend $P$ to an induced $S$-path on at least $t$ vertices.
    This is a contradiction.
    
    Consequently, we can apply induction to $G'$, $S'$, $k$, and $t-2$.
    It follows that $\td(G',S') \leq f(k,t-2)$.
    Let $X \subset V(G')$ be a set given by \cref{lem:apply-elimination}, i.e.\ $S' \subset X \subset V(G')$ and
    \[\td(G',S) \leq \td(G',S') + \td(G' - X,S \setminus X) \leq f(k,t-2) + \td(G' - X,S \setminus X).\]
    For every $B \in \calB(G')$, let $v_B$ be such that $\tdtwo(B) = \tdtwo(B - v_B) + 1$.
    Let $X' = \{v_B : B \in \calB(G')\}$.
    %\jedrzej{Here, if $X' \subset X$, we are super happy and can get a bound matching the lower bound. But I don't see why this is true.}
    Since $\calA(G') = S' \subset X$, every component of $G'-X$ contains at most one vertex of $X'$.
    This implies
    \[\td(G' - X, S \setminus X) \leq \td(G' - (X \cup X'), S \setminus (X \cup X')) + 1.\]
    Let $B \in \calB(G')$ be such that $\tdtwo(G') = \tdtwo(B)$.
    We have
    \[k \geq \tdtwo(G) \geq \tdtwo(G') = \tdtwo(B) = \tdtwo(B - v_B) + 1 \geq \tdtwo(G' - (X \cup X')).\]
    By induction applied to $G' - (X \cup X')$, $S - (X \cup X')$, $k-1$, and $t$, we have 
    \[\td(G' - (X \cup X'),S - (X \cup X')) \leq f(k-1,t).\]
    Altogether, we obtain $\td(G,S) \leq f(k,t-2) + f(k-1,t) + 1 = f(k,t)$, as desired.
    This ends the proof.
\end{proof}

\section{Treedepth of \texorpdfstring{$P_4$}{P4}-free and \texorpdfstring{$P_5$}{P5}-free graphs} \label{sec:short-paths}

We begin by studying the structure of forests of blocks of $P_4$-free and $P_5$-free graphs.
To this end, we need the following lemma.

\begin{lemma} \label{lem:diam}
    For every integer $t \geq 3$, for every connected $P_t$-free graph $G$, we have
    \[\diam(\btree(G)) \leqslant 2(t-3).\]
\end{lemma}

\begin{proof}
Let $t$ be an integer with $t \geq 3$ and $G$ be a connected $P_t$-free graph.
If $G$ has no edges, then the assertion is clear; thus, assume that $G$ has some edges.
Let $S$ be a path of maximum length in $\btree(G)$. 
Recall that $|V(S)|$ is odd.
If $|V(S)| = 1$, then the statement of the lemma is trivially true.
Therefore, we may assume that $|V(S)| \geqslant 3$ and let
\[S = B_1v_1B_2v_2\cdots B_mv_mB_{m+1}.\]
In particular, $v_i \in \calA(G)$ for each $i \in [m]$ and $B_i \in \calB(G)$ for each $i \in [m+1]$.
For each $i \in [m-1]$, let $P_{i+1}$ be a shortest path in $G$ connecting $v_i$ and $v_{i+1}$.
Note that $v_i$ and $v_{i+1}$ lie in the same block of $G$ (namely, $B_{i+1}$), hence, $P_{i+1}$ also lies in $B_{i+1}$.
Since $B_1 \neq B_{m+1}$ and $G$ is connected, both $B_1$ and $B_{m+1}$ have at least two vertices.
Let $v_0$ be a neighbor of $v_1$ in $B_1$ and let $v_{m+1}$ be a neighbor of $v_m$ in $B_{m+1}$.
We define $P_1 = v_0v_1$ and $P_{m+1} = v_mv_{m+1}$.
Note that the paths $P_i$ and $P_j$ for each distinct $i,j \in [m+1]$ are internally disjoint since they lie in different blocks (except possibly endpoints).
Let $P$ be the concatenation of $P_1,\dots,P_{m+1}$.
Observe that $P$ is also an induced path in $G$.
Indeed, there is no shortcut between two vertices in $P_i$ for each $i \in [m+1]$ since $P_i$ is a shortest path in $G$ and there is no shortcut between a vertex in $P_i$ and a vertex in $P_j$ for distinct $i,j \in [m+1]$ since each of the paths composing $P$ lies in a different block of $G$.
By definition, $v_0,\dots,v_{m+1}$ are a collection of distinct vertices of $P$, hence, $|V(P)| \geq m+2$.
In particular, $m+2 < t$ as $G$ is $P_t$-free.
Finally,
\[\diam(\btree(G)) = |V(S)|-1 = 2m \leq 2(t-3).\qedhere\]
\end{proof}

Let $G$ be a connected $P_5$-free graph with at least two vertices.
By \cref{lem:diam}, $\diam(\btree(G)) \leq 4$, and recall that $\diam(\btree(G))$ is even.
Therefore, there are three cases to consider.
If $\diam(\btree(G)) = 0$, then $G$ is a block.
Assume that $\diam(\btree(G)) = 2$.
It follows that $\btree(G)$ is a star.
Since all leaves of $\btree(G)$ are blocks of $G$, there is a unique cut-vertex of $G$.
Moreover, this cut-vertex is adjacent in $\btree(G)$ to all the blocks of $G$.
In particular, $G$ consists of a bunch of blocks sharing one cut-vertex.
Finally, assume that $\diam(\btree(G)) = 4$.
We claim that there is only one block in $G$ that is not a leaf in $\btree(G)$.
Indeed, otherwise, there is a path in $\btree(G)$ of length at least $6$, which is a contradiction.
In particular, $G$ consists of one \defin{central block} such that all other blocks share a cut-vertex with this central block.
See \cref{fig:P5-cases}.
We will use the above observations implicitly in the following proofs.
Note that for a $P_4$-free graph $G$, by \cref{lem:diam}, either $\diam(\btree(G)) = 0$ or $\diam(\btree(G)) = 2$.
In the next lemma, we discuss the structure of $P_4$-free and $P_5$-free graphs depending on the diameter of their forests of blocks.

\vbox{
\begin{lemma}\label{lem:p4-p5}
    Let $G$ be a connected graph.
    \begin{enumerate}
        \item If $G$ is $P_4$-free and $\diam(\btree(G)) = 2$, then the unique cut-vertex of $G$ is an apex in $G$. \label{item:p4-2}
        \item If $G$ is $P_5$-free and $\diam(\btree(G)) = 2$, then there is at most one block $B$ of $G$ such that the unique cut-vertex of $G$ is not an apex in $B$.\label{item:p5-2}
        \item If $G$ is $P_5$-free and $\diam(\btree(G)) = 4$, then the unique cut-vertex of $G$ in $V(B)$ is an apex in $B$ for every non-central block $B \in \calB(G)$.\label{item:p5-4i}
        \item If $G$ is $P_5$-free and $\diam(\btree(G)) = 4$, then $G[\calA(G)]$ is a complete graph.\label{item:p5-4ii}
    \end{enumerate}
\end{lemma}
}
\begin{proof}
    First, assume that $G$ is $P_4$-free with $\diam(\btree(G)) = 2$ and let $v$ be the unique cut-vertex of $G$.
    Suppose to the contrary that there is $u \in V(G-v)$ not adjacent to $v$ in $G$.
    Let $B$ be the block of $G$ containing $u$.
    Since $\diam(\btree(G)) = 2$, there is $D \in \calB(G) \setminus \{B\}$ and $w \in V(D)$ adjacent to $v$ in $G$.
    Let $P$ be a shortest path from $v$ to $u$ in $B$.
    Note that $P$ has at least one internal vertex.
    It follows that $R = uPvw$ is an induced path in $G$ and $R$ has at least $4$ vertices, which is a contradiction that yields~\ref{item:p4-2}.

    Next, assume that $G$ is $P_5$-free with $\diam(\btree(G)) = 2$ and let $v$ be the unique cut-vertex of $G$.
    Suppose to the contrary that there are $B_1,B_2 \in \calB(G)$ such that $v$ is not an apex in either $B_1$ or $B_2$.
    Fix $i \in [2]$.
    Let $u_i \in V(B_i - v)$ be non-adjacent to $v$ in $G$, and let $P_i$ be a shortest path in $G$ connecting $u_i$ and $v$.
    Note that $P_i$ has at least one internal vertex.
    Observe that $R = u_1P_1vP_2u_2$ is an induced path in $G$.
    Moreover, $R$ has at least $5$ vertices, which is a contradiction that yields~\ref{item:p5-2}.

    Assume that $G$ is $P_5$-free with $\diam(\btree(G)) = 4$ and let $B \in \calB(G)$ be a non-central block.
    Additionally, let $v$ be the unique cut-vertex of $G$ in $V(B)$ and let $C$ be the central block of $G$.
    Let $BvCwD$ be a path in $\btree(G)$. %where $D$ is a block of $G$ distinct from $B$ and $C$, and $w$ is a cut-vertex of $G$ common to $C$ and $D$.
    Suppose to the contrary that there is $u \in V(B - v)$ not adjacent to $v$ in $G$.
    Let $P$ be a shortest path in $G$ connecting $u$ and $v$.
    Note that $P$ has at least one internal vertex.
    Let $Q$ be a shortest path in $G$ connecting $v$ and $w$, and let $x \in V(D-w)$ be adjacent to $w$.
    Observe that $R = uPvQwx$ is an induced path in $G$.
    Moreover, since $v \neq w$, $R$ has at least $5$ vertices, which is a contradiction that yields~\ref{item:p5-4i}.

    Finally, assume that $G$ is $P_5$-free with $\diam(\btree(G)) = 4$ and let $C$ be the central block of $G$.
    Suppose to the contrary that there are non-adjacent $u_1,u_2 \in \calA(G)$.
    Note that $u_1,u_2 \in V(C)$.
    Let $P$ be a shortest path in $C$ connecting $u_1$ and $u_2$.
    Note that $P$ has at least one internal vertex.
    For each $i \in [2]$, let $v_i$ be a neighbor of $u_i$ that does not lie in $C$.
    Observe that $R = v_1u_1Pu_2v_2$ is an induced path in $G$.
    Moreover, $R$ has at least $5$ vertices, which is a contradiction that yields~\ref{item:p5-4ii}.
\end{proof}

\begin{figure}[tp]
    \centering
    \includegraphics{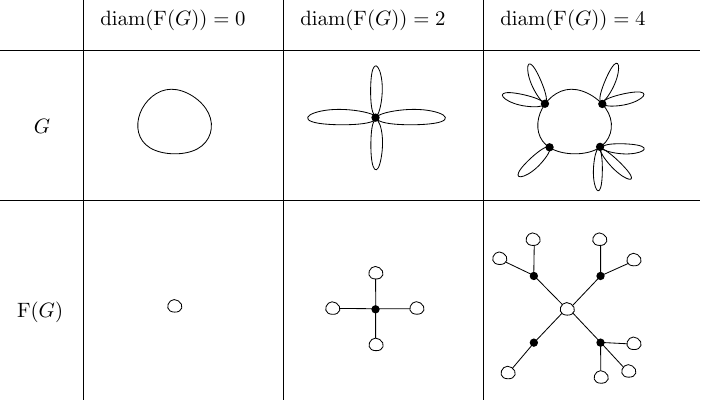}
    \caption{Examples of block structure of $P_5$-free graphs and their forests of blocks depending on $\diam(\btree(G))$.}
    \label{fig:P5-cases}
\end{figure}

\begin{proof}[Proof of \cref{thm:p5free-upper}]
    First, we prove~\ref{item:p4}.
    The proof is by induction on $|V(G)|$.
    If $G$ is the null graph, then the statement holds.
    If $G$ has more than one component, then $\td(G) = \td(C)$ for some component $C$ of $G$, and since $|V(C)| < |V(G)|$, by the induction hypothesis, $\td(G) = \td(C) \leq \tdtwo(C) \leq \tdtwo(G)$.
    Thus, we assume that $G$ is connected.
    Recall that $\diam(\btree(G)) \in \{0,2\}$ as $G$ is $P_4$-free.
    If $\diam(\btree(G)) = 0$, then $G$ is a block.
    In this case, there exists $v \in V(G)$ such that $\tdtwo(G) = \tdtwo(G - v) + 1$, and so, by the induction hypothesis, $\td(G) \leq \td(G-v) + 1 = \tdtwo(G-v)+1 = \tdtwo(G)$.
    If $\diam(\btree(G)) = 2$, then $G$ has a unique cut-vertex, which, by \cref{lem:p4-p5}.\ref{item:p4-2}, is an apex in $G$.
    By \cref{lem:apex-td2} and the induction hypothesis, $\td(G) = \td(G-v)+1 \leq \tdtwo(G-v)+1 = \tdtwo(G)$.
    This completes the proof of~\ref{item:p4}.

    Next, we prove~\ref{item:p5} by induction on $\tdtwo(G) + |V(G)|$.
    If $G$ is the null graph, then the statement holds.
    If $\tdtwo(G) = 1$, then $G$ has no edges, and thus $\td(G) \leqslant 1 = \binom{\tdtwo(G)+1}{2}$. 
    Thus, assume ${\tdtwo(G) \geqslant 2}$. 
    If $G$ has more than one component, then $\td(G) = \td(C)$ for some component $C$ of $G$.
    Since $|V(C)| < |V(G)|$, by the induction hypothesis, 
    \[\td(G) = \td(C) \leq \binom{\tdtwo(C)+1}{2} \leq \binom{\tdtwo(G)+1}{2} .\]
    Therefore, we may assume that $G$ is connected. 
    Recall that $\diam(\btree(G)) \in \{0,2,4\}$.

    First, suppose that $\diam(\btree(G)) = 0$.
    It follows that $G$ is a block.
    There exists $v \in V(G)$ such that $\tdtwo(G) = \tdtwo(G - v) + 1$. 
    By the induction hypothesis, 
    \[\td(G) \leq \td(G-v) + 1 \leq \binom{\tdtwo(G - v)+1}{2}+1 \leq \binom{\tdtwo(G - v)+2}{2} = \binom{\tdtwo(G)+1}{2}.\]

    Next, suppose that $\diam(\btree(G)) = 2$.
    Let $v$ be the unique cut-vertex of $G$.
    By \cref{lem:p4-p5}.\ref{item:p5-2}, $v$ is an apex in all but at most one block of $G$. 
    Let $D$ be such that $v$ is an apex in every $B \in \calB(G) \setminus \{D\}$.
    Let $u \in V(D)$ such that $\tdtwo(D) = \tdtwo(D - u) + 1$. 
    Applying the induction hypothesis to $G - u - v$, we obtain $\td(G - u - v) \leqslant \binom{\tdtwo(G - u - v)+1}{2}$. 
    Note that $\tdtwo(G - u - v) \leqslant \tdtwo(G) - 1$ by \cref{lem:apex-td2} and the choice of $u$. 
    This gives us
    \[
    \td(G) \leqslant \td(G - u - v) + 2 \leqslant \binom{\tdtwo(G - u - v) + 1}{2} + 2 \leqslant \binom{\tdtwo(G)}{2} + 2 \leqslant \binom{\tdtwo(G) + 1}{2}.
    \]
    %The final inequality follows from $\tdtwo(G) \geqslant 2$.

    Finally, suppose that $\diam(\btree(G)) = 4$. 
    Let $C \in \calB(G)$ be the central block of $G$ and let $v \in V(C)$ be such that $\tdtwo(C) = \tdtwo(C-v) + 1$. 
    Let $A = \calA(G) \cup \{v\}$.
    For every $B \in \calB(G) \setminus \{C\}$, let $v_B$ be the unique cut-vertex of $G$ in $V(B)$.
    Observe that for every $B \in \calB(G) \setminus \{C\}$, by \cref{lem:p4-p5}.\ref{item:p5-4i}, $v_B$ is an apex in $B$, and so, by \cref{lem:apex-td2}, $\tdtwo(B) = \tdtwo(B-v_B)+1$.
    By the induction hypothesis applied to $C-v$ and $B-v_B$ for each $B \in \calB(G) \setminus \{C\}$,
    \begin{align*}
        \td(G) \leq \td(G - A) + |A| & \leqslant \max\left(\td(C - v), \max_{B \in \calB(G) \setminus \{C\}} \td(B - v_B)\right) + |A| \\
        &\leqslant \max\left(\binom{\tdtwo(C - v)+1}{2}, \max_{B \in \calB(G) \setminus \{C\}} \binom{\tdtwo(B - v_B)+1}{2}\right) + |A| \\
        & \leqslant \max_{B \in \calB(G)} \binom{\tdtwo(B)}{2} + |A| \leqslant \binom{\tdtwo(G)}{2} + |A|.
    \end{align*}
    To conclude, it suffices to prove that $|A| \leqslant \tdtwo(C) \leq \tdtwo(G)$. 
    Note that $\calA(G) \subset V(C)$.
    By \cref{lem:p4-p5}.\ref{item:p5-4ii}, $\tdtwo(G[\calA(G)]) = |\calA(G)|$. 
    If $v \in \calA(G)$, then $A = \calA(G)$ and thus indeed $|A| = \tdtwo(G[\calA(G)]) \leqslant \tdtwo(C)$. 
    On the other hand, if $v \notin \calA(G)$, then $\calA(G) \subseteq V(C - v)$, and thus $|\calA(G)| = \tdtwo(G[\calA(G)]) \leqslant \tdtwo(C - v) = \tdtwo(C) - 1$, which gives us $|A| \leqslant \tdtwo(C)$. 
\end{proof}

\section{A lower bound for \texorpdfstring{\cref{thm:main}}{main theorem}} \label{sec:lower-bound}

Grohe et al.~\cite{Grohe15} proved that the bound in \cref{wcol:bd:tw} is sharp.
Their construction appeared later to be of great importance to the topic of weak colouring numbers~\cite{wcols}.
We simplify and adjust the construction slightly to serve as a lower bound for \cref{thm:main}.

We construct recursively graphs $G_{r,k}$ for all positive integers $r$ and $k$.
First, for all positive integers $r$ and $k$, we set $G_{r,1}$ to be a one-vertex graph and $G_{1,k}$ to be a complete graph on $k$ vertices.
Now, given integers $r \geq 2$ and $k \geq 2$, and having defined $G_{r-1,k}$ and $G_{r,k-1}$, we do the following.
The graph $G_{r,k}$ is obtained from $G_{r-1,k}$ by adding for each vertex $v \in V(G_{r-1,k})$, a copy of $G_{r,k-1}$ and connecting $v$ to every vertex in the copy.
See \cref{fig:p5construction} for an example.

Given integers $t \geq 3$ and  $k \geq 1$, we claim that the graph $G_{r,k}$ with $r = \lfloor (t-1) \slash 2 \rfloor$ witnesses \cref{thm:lower-bound}.
To this end, we prove the following statement.

\begin{lemma}\label{lem:lower-bound-technical}
    For all positive integers $r$ and $k$, we have
    \begin{enumerate}
        \item $\tdtwo(G_{r,k}) \leq k$, \label{item:tdtwo}
        \item $G_{r,k}$ is $P_{2r+1}$-free, \label{item:p-free}
        \item $\td(G_{r,k}) \geq \binom{r+k-1}{r}$. \label{item:td}
    \end{enumerate}
\end{lemma}
\begin{proof}
Item~\ref{item:tdtwo} follows from a simple induction on $r+k$.
The base cases of the construction satisfy the assertion.
Next, note that when $r,k \geq 2$, then the blocks of $G_{r,k}$ are either blocks of $G_{r-1,k}$ or $G_{r,k-1}$ with an additional vertex adjacent to all the other vertices.
By induction, $\tdtwo(G_{r-1,k}) \leq k$.
By the definition of $2$-treedepth, the $2$-treedepth of the other blocks is at most $\tdtwo(G_{r,k-1}) + 1$, which is at most $k$ by induction.
This way, we obtain~\ref{item:tdtwo}.

Item~\ref{item:p-free} again follows from induction on $r+k$.
Again, the base cases of the construction satisfy the assertion.
Suppose that $r,k \geq 2$.
Let $P$ be a longest induced path in $G_{r,k}$.
If $P$ is contained in a copy of $G_{r,k-1}$, then by induction $|V(P)| < 2r+1$.
Otherwise, $P$ intersects $G_{r-1,k}$.
Let $X$ be the set of vertices in this intersection.
By construction, $P[X]$ is connected.
Therefore, by induction, $|V(P[X])| < 2r-1$.
Now, since the vertices of $G_{r-1,k}$ are adjacent to all vertices in respective copies of $G_{r,k-1}$, $P$ can have only two vertices outside of $G_{r-1,k}$.
We conclude that $|V(P)| < 2r+1$.

Finally, to prove~\ref{item:td}, it is handy to have the following abstract claim.

\begin{claim}\label{claim:type-a}
    Let $H$ and $H'$ be connected graphs.
    A graph $G$ is of type $a(H,H')$ if it is obtained from $H$ by adding, for every vertex $v$ of $H$, a copy $H_v$ of $H'$ such that every vertex of the copy is adjacent to $v$.
    For every graph $G$ of type $a(H,H')$, we have
    \[\td(G) \geq \td(H) + \td(H').\]
\end{claim}
\begin{proofclaim}
    We proceed by induction on $\td(H) + \abs{V(H)}$.
    If $\td(H) = 1$, then $H$ is a single vertex, and a graph $G$ of type $a(H,H')$ consists of a copy of $H'$ and an apex.
    Thus, by \cref{lem:apex-td2}, $\td(G) = 1 + \td(H')$.
    Now, assume that $\td(H) \geq 2$ and let $G$ be of type $a(H,H')$.
    Let $u \in V(G)$ be such that $\td(G) = \td(G-u) + 1$.
    Let $v \in V(H)$ be such that $u \in V(H_v) \cup \set{v}$.
    In particular, $\td(H-v) \geq \td(H) - 1$.
    Note that $G - \set{u, v}$ has a component $C$ of type $a(K,H')$ where $K$ is a component of $H-v$ with $\td(K) = \td(H-v)$.
    By induction applied to $C$, we obtain,
    \begin{align*}
        \td(G) = \td(G - u) + 1 &\geq \td(C) + 1 \geq \td(H-v) + \td(H') + 1 \geq \td(H) + \td(H').
    \end{align*}
    This completes the proof of the claim.
\end{proofclaim}

Note that for integers $r \geq 2$ and $k \geq 2$, the graph $G_{r,k}$ is of type $a(G_{r-1,k},G_{r,k-1})$.
Therefore, by \cref{claim:type-a}, 
\[\td(G_{r,k}) \geq \td(G_{r-1,k}) + \td(G_{r,k-1}).\]
Additionally, for all positive integers $r$ and $k$, we have $\td(G_{r,1}) = 1 = \binom{r+1-1}{r}$ and $\td(G_{1,k}) = k = \binom{1+k-1}{1}$.
Now, a simple inductive computation shows that $\td(G_{r,k}) \geq \binom{r+k-1}{r}$ for all positive integers $r$ and $k$, and so,~\ref{item:td} holds.
\end{proof}

\begin{figure}[tp]
    \centering
    \includegraphics{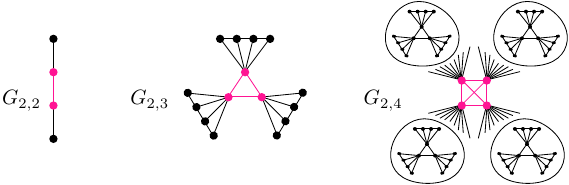}
    \caption{First steps of the construction of graphs $G_{2,k}$. Note that these graphs witness the tightness of the bound on $g(5,k)$ in \cref{thm:p5free-upper}. Pink vertices are the vertices of $G_{1,k}$, i.e.\ complete graphs on $k$ vertices.}
    \label{fig:p5construction}
\end{figure}

\section{Forcing long induced paths}\label{sec:forcing:proofs}

First, we define pathwidth formally.
The definition that we give is not the most standard one.
An \defin{interval graph} is an intersection graph of finite intervals of the real line.
We assume that interval graphs are given with their interval representations.
The \defin{pathwidth} of a graph $G$ is the minimum non-negative integer $p$ such that there exists an interval graph $I$ of clique number $p+1$ and $G$ is a subgraph of $I$.
In particular, the pathwidth of an interval graph is its clique number minus one.

\subsection{Construction}

In this subsection, we give a construction of graphs with small pathwidth and $2$-treedepth, long paths as subgraphs but no very long paths as induced subgraphs.
More precisely, we prove the following statement.

\begin{theorem}\label{thm:pwtd2-lower}
    For all positive integers $\ell$ and $k$, there exists a graph $G_{\ell, k}$ such that 
    \begin{enumerateAlpha}
        \item $G_{\ell, k}$ contains a path of order $\binom{\ell + k - 1}{k}$\textup{;}
        \item $\pw(G_{\ell, k}) \leq k$ and $\tdtwo(G_{\ell, k}) \leq k + 1$\textup{;}
        \item $G_{\ell, k}$ is $P_{\ell+1}$-free.
    \end{enumerateAlpha}
\end{theorem}

Assume that \cref{thm:pwtd2-lower} holds.
We argue that \cref{thm:pw,thm:td2-lower} follow.
Note that $\binom{\ell + k - 1}{k} = \frac{(\ell + k - 1) \dotsm (\ell + 1) \ell}{k!} \geq \ell^k/k!$ for all positive integers $\ell$ and $k$.

\begin{proof}[Proof of \cref{thm:pw}]
    For a given positive integer $k$, we set $c_k = k!$.
    Now, given a positive integer $n$, we set $\ell$ to be such that $\binom{\ell+k-1}{k} \geq n > \binom{\ell+k-2}{k}$.
    By \cref{thm:pwtd2-lower}, $G_{\ell,k}$ contains $P_n$ as a subgraph, $\pw(G_{\ell,k}) \leq k$, and $G_{\ell,k}$ is $P_{\ell+1}$-free.
    Since $n > \binom{\ell+k-2}{k} \geq \frac{(\ell-1)^k}{k!}$, we obtain $k! \cdot n^{1 \slash k} \geq \ell$.
    This shows that $G_{\ell,k}$ witnesses \cref{thm:pw} for $k$ and $n$ and ends the proof.
\end{proof}

\begin{proof}[Proof of \cref{thm:td2-lower}]
    For a given positive integer $k$, we set $c_k = (k-1)!$.
    Now, given a positive integer $n$, we set $\ell$ to be such that $\binom{\ell+k-2}{k-1} \geq n > \binom{\ell+k-3}{k-1}$.
    By \cref{thm:pwtd2-lower}, $G_{\ell,k-1}$ contains $P_n$ as a subgraph, $\tdtwo(G_{\ell,k-1}) \leq k$, and $G_{\ell,k-1}$ is $P_{\ell+1}$-free.
    Since $n > \binom{\ell+k-3}{k-1} \geq \frac{(\ell-1)^{k-1}}{(k-1)!}$, we obtain $(k-1)! \cdot n^{1 \slash (k-1)} \geq \ell$.
    This shows that $G_{\ell,k-1}$ witnesses \cref{thm:td2-lower} for $k$ and $n$ and ends the proof.
\end{proof}

\begin{proof}[Proof of \cref{thm:pwtd2-lower}]
    We will iteratively define the graphs $G_{\ell, k}$ as well as a special root vertex $r_{\ell, k} \in V(G_{\ell,k})$. 
    For every positive integer $\ell$, $G_{\ell,1}$ is a path of order $\ell$ and $r_{\ell, 1}$ is one of its endpoints. 
    Having defined $G_{\ell, k}$ and $r_{\ell, k}$ for some fixed positive integer $k$ and all positive integers $\ell$, we define $G_{\ell,k+1}$ and $r_{\ell, k + 1}$ for a positive integer $\ell$ as follows.
    Take the disjoint union of the graphs $G_{1, k}$, $G_{2, k}$, \ldots, $G_{\ell, k}$, and for every $s \in [\ell-1]$, add an edge between $r_{s,k}$ and every vertex of $G_{s+1,k}$.
    Finally, we set as $r_{\ell,k+1}$, the vertex $r_{\ell,k}$ from the copy of $G_{\ell,k}$ that we used.
    We remark in passing that $G_{2, k}$ is a complete graph on $k + 1$ vertices.
    Two examples are shown in \cref{fig:Glkconstruction}.

    \begin{figure}[tp]
    \centering
        \begin{subfigure}[b]{0.49\textwidth}
            \centering
            \begin{tikzpicture}
                \pgfmathtruncatemacro\n{4}
                \foreach \i in {1, ..., \n}{
                    \pgfmathtruncatemacro{\in}{\i - 1}
                    \foreach \j in {1, ..., \i}{
                        \tkzDefPoint(1.5*\i,\j){p_{\i,\j}}
                        \ifthenelse{\i > 1}{
                            \tkzDrawSegment[lightgray](p_{\in,1},p_{\i,\j})
                        }
                        {}
                    }
                }
        
                \foreach \i in {1, 2, ..., \n}{
                    \tkzLabelPoint[below](p_{\i,1}){$r_{\i,1}$}
                    \foreach \j in {1, ..., \i}{
                        \tkzDrawSegment(p_{\i,1},p_{\i,\i})
                        \tkzDrawPoint(p_{\i,\j})
                    }
                }
            \end{tikzpicture}
            \caption{$G_{4, 2}$}
        \end{subfigure}%
        \begin{subfigure}[b]{0.49\textwidth}
            \centering
            \begin{tikzpicture}
                \pgfmathtruncatemacro\n{3}
                \foreach \i in {1, ..., \n}{
                    \pgfmathtruncatemacro{\in}{\i - 1}
                    \foreach \j in {1,...,\i}{
                        \foreach \k in {1, ..., \j}{
                            \tkzDefPoint(0.5*\i*\i+\j,-0.5*\i+\k){p_{\i,\j,\k}}
                            \ifthenelse{\i > 1}{
                                \tkzDrawSegment[lightgray](p_{\in,\in,1},p_{\i,\j,\k})
                            }
                            {}
                        }
                    }
                }

                \foreach \i in {1,...,\n}{
                    \tkzLabelPoint[below](p_{\i,\i,1}){$r_{\i,2}$}
                    \foreach \j in {1,...,\i}{
                        \pgfmathtruncatemacro{\jn}{\j - 1}
                        \foreach \k in {1,...,\j}{
                            \tkzDrawPoint(p_{\i,\j,\k})
                            \ifthenelse{\j > 1}{
                                \tkzDrawSegment(p_{\i,\jn,1},p_{\i,\j,\k})
                            }
                            {}
                        }
                        \tkzDrawSegment(p_{\i,\j,1},p_{\i,\j,\j})
                    }
                }
            \end{tikzpicture}
            \caption{$G_{3, 3}$}
        \end{subfigure}
        \caption{}
        \label{fig:Glkconstruction}
    \end{figure}

    In order to prove the theorem, we prove a stronger statement by induction on $k$.
    For every positive integer $\ell$,
    \begin{enumerateAlphaPrim}
        \item $G_{\ell, k}$ has $\binom{\ell + k - 1}{k}$ vertices and contains a Hamiltonian path that ends at $r_{\ell, k}$; \label{ind:Hamiltonian}
        \item $\pw(G_{\ell, k}) \leq k$ and $\tdtwo(G_{\ell, k}) \leq k + 1$;\label{ind:pw}
        \item $G_{\ell,k}$ is $P_{\ell+1}$-free;\label{ind:free}
        \item $G_{\ell, k}$ is an interval graph of clique number at most $k+1$ where the interval corresponding to $r_{\ell, k}$ has the rightmost right endpoint.\label{ind:interval}
    \end{enumerateAlphaPrim}

    Note that \ref{ind:Hamiltonian}-\ref{ind:interval} all hold for $\ell = 1$ or $k = 1$.
    Suppose they hold for some fixed value of $k$ (and all values of $\ell$). 
    We will now prove \ref{ind:Hamiltonian}-\ref{ind:interval} for $k + 1$. Let $\ell \geq 2$.
    
    First, the number of vertices in $G_{\ell, k + 1}$ is the sum of the number of vertices in $G_{1, k}$, \ldots, $G_{\ell, k}$ which is
    \[
        \sum_{s = 1}^{\ell} \binom{s + k - 1}{k} = \binom{\ell + k}{k + 1}.
    \]
    Additionally, for each $s \in [\ell]$, $G_{s, k}$ has a Hamiltonian path $P_s$ that ends at $r_{s, k}$. 
    It follows that $P_1 P_2 \dots P_\ell$ is a Hamiltonian path of $G_{\ell, k + 1}$ that ends at $r_{\ell, k + 1}$, which gives~\ref{ind:Hamiltonian}.

    Next, we prove~\ref{ind:interval}.
    For each $s \in [\ell]$, $G_{s, k}$ is an interval graph, say, represented by a collection of intervals $\calR_s$ where the interval corresponding to $r_{s, k}$ has the rightmost right endpoint. 
    Place the collections $\calR_1$, \ldots, $\calR_\ell$ in that order on the real line so that intervals from different collections are disjoint. 
    This is a realisation of the disjoint union of $G_{1, k}$, \dots, $G_{\ell, k}$ as an interval graph. 
    For each $s \in [\ell-1]$, take the interval corresponding to $r_{s, k}$ and extend it to the right so that it intersects every interval in $\calR_{s + 1}$ but no interval in $\calR_{s + 2}$.
    Additionally, extend the interval corresponding to $r_{\ell, k}$ so that its right endpoint is rightmost. 
    Since the interval corresponding to $r_{s, k}$ has the rightmost right endpoint within $\calR_s$, extending this interval only adds the edges between $r_{s, k}$ and each vertex of $G_{s + 1, k}$.
    Therefore, $G_{\ell, k + 1}$ is an interval graph where the interval corresponding to $r_{\ell, k+1}$ has the rightmost right endpoint.
    Moreover, by construction,  $G_{\ell, k + 1}$ has the clique number one greater than the maximum of the clique number of $G_{1, k}$, \dots, $G_{\ell, k}$.
    Thus, applying induction,~\ref{ind:interval} holds.
    
    The fact that $\pw(G_{\ell,k+1}) \leq k+1$ follows from~\ref{ind:interval}.
    To see that $\tdtwo(G_{\ell,k+1}) \leq k+2$ note that the blocks of $G_{\ell,k+1}$ are of the form $B_s = G_{\ell,k+1}[\{r_{s,k}\} \cup V(G_{s+1,k})]$ for $s \in [\ell-1]$.
    Note that $r_{s,k}$ is an apex in $B_s$, hence, by \cref{lem:apex-td2}, $\tdtwo(B_s) = 1 + \tdtwo(G_{s+1,k})$.
    By induction, we obtain
    \[
        \tdtwo(G_{\ell, k + 1}) = \max_{s \in [\ell-1]} \tdtwo(B_s) \leq 1 + \max_{s \in [\ell-1]} \tdtwo(G_{s + 1, k}) \leq 1+(k+1) = k+2.
    \]
    This shows that~\ref{ind:pw} holds.

    It remains to prove~\ref{ind:free}.
    We remark that showing $P_{3\ell+1}$-freeness of $G_{\ell,k+1}$ is a simple exercise.
    However, for the sake of precision, we show that $G_{\ell,k+1}$ is $P_{\ell+1}$-free, which is best possible, but the proof is slightly more involved.
    Let $Q$ be an induced path in $G_{\ell, k + 1}$. 
    First, suppose that $Q$ does not contain any root vertex $r_{s, k}$. Then $Q$ lies entirely inside $G_{s, k}$ for some $s \in [\ell]$ and so has at most $s \leq \ell$ vertices, by induction. 
    Next, suppose that $Q$ contains exactly one root, $r_{s, k}$ for some $s \in [\ell-1]$. 
    This implies that $Q$ lies inside $G_{s, k} \cup G_{s + 1, k}$. 
    Assume that $Q$ intersects $G_{s+1,k}$ (we have already considered the other case).
    Since $r_{s, k}$ dominates $G_{s + 1, k}$ and $Q$ is induced, any vertex of $G_{s + 1, k}$ in $Q$ is an endpoint of $Q$. 
    This implies that $Q$ contains at most two vertices from $G_{s + 1, k}$. 
    %If $Q$ contains no vertex of $G_{s + 1, k}$, then $Q \subset G_{s, k}$ and so $Q$ has at most $s \leq \ell$ vertices, by induction. 
    If $Q$ contains one vertex $v$ from $G_{s + 1, k}$, then $Q - v$ is an induced path in $G_{s, k}$, and so, $Q-v$ has at most $s$ vertices by induction.
    It follows that $Q$ has at most $s + 1 \leq \ell$ vertices.
    Finally, if $Q$ contains two vertices $u$ and $v$ from $G_{s + 1, k}$, then $Q = u r_{s, k} v$ is a path on 3 vertices. 
    But $G_{1, k}$ and $G_{2, k}$ (complete graphs on resp.\ $1$ and $k+1$ vertices) do not contain an induced $P_3$, so $s \geq 2$, and in particular, $\ell \geq 3$. 
    In all cases $Q$ has at most $\ell$ vertices and so $G_{\ell,k+1}$ is $P_{\ell+1}$-free as desired.
    Finally, suppose that $Q$ contains at least two root vertices. Let $s^\ast = \min\set{s \in [\ell] : r_{s, k} \in V(Q)}$ and $t^\ast = \max\set{s \in [\ell] : r_{s, k} \in V(Q)}$.
    Note that any path from $r_{s^\ast, k}$ to $r_{t^\ast, k}$ passes through each of $\set{r_{s, k} : s^\ast < s < t^\ast}$, and so, all these vertices are in $Q$. 
    Furthermore, as $Q$ is induced and $P = r_{s^\ast, k} r_{s^\ast + 1, k} \dots r_{t^\ast, k}$ is a path in $G_{\ell, k + 1}$, we must have $Q = Q_s r_{s^\ast, k} P r_{t^\ast, k} Q_t$ where $Q_s$ is a path ending at $r_{s^\ast, k}$ and $Q_t$ is a path starting at $r_{t^\ast, k}$. 
    We first consider $Q_t$. 
    Note that $r_{t^\ast + 1, k} \notin V(Q_t)$, by maximality of $t^\ast$, and $r_{t^\ast - 1, k} \notin V(Q_t)$ since it is already in $P$. 
    Thus $Q_t \subset G_{t^\ast, k} \cup G_{t^\ast + 1, k}$. But $r_{t^\ast - 1, k} \in V(P)$ dominates $G_{t^\ast, k}$ and so $V(Q_t) \subset \set{r_{t^\ast, k}} \cup V(G_{t^\ast + 1, k})$. Finally, $r_{t^\ast, k}$ dominates $G_{t^\ast + 1, k}$ and so $Q_t$ contains at most one vertex of $G_{t^\ast + 1, k}$.
    It follows that $Q_t$ has at most 2 vertices (and if it is 2, then $t^\ast \leq \ell - 1$). 
    Overall, $P r_{t^\ast, k} Q_t$ has at most $\ell - s^\ast+1$ and so it remains to show that $Q_s$ has fewer than $s^\ast+1$ vertices. 
    Note that $r_{s^\ast - 1, k} \notin V(Q_s)$ by minimality of $s^\ast$ and $r_{s^\ast + 1, k} \notin V(Q_s)$ since it is already in $P$. 
    Thus $Q_s \subset G_{s^\ast, k} \cup G_{s^\ast + 1, k}$. 
    Since $r_{s^\ast, k}$ dominates $G_{s^\ast + 1, k}$ and $Q$ is induced, any vertex of $Q_s$ in $G_{s^\ast + 1, k}$ is an endpoint of $Q_s$. 
    Thus, if $Q_s$ contains a vertex $v$ of $G_{s^\ast + 1, k}$, then $Q_s = vr_{s^\ast, k}$. 
    Since $Q$ is induced, $v$ is not adjacent to $r_{s^\ast + 1, k}$.
    It follows that $G_{s^\ast + 1, k}$ is not a complete graph, and so, $s^\ast + 1 \geq 3$. 
    Therefore, in this case, $Q_s$ has fewer than $s^\ast+1$ vertices, as required. 
    Otherwise, $Q_s \subset G_{s^\ast, k}$ and so, by induction, the desired statement follows.
    This shows~\ref{ind:free} and ends the proof.
\end{proof}

\subsection{Upper bound}

We now prove \cref{thm:td2-upper}, which we repeat for convenience.

\tdupper*

\begin{proof}[Proof of \cref{thm:td2-upper}]
    If $n = 1$, then the proof is immediate, so assume $n \geq 2$.
    The proof is by induction on $k$. 
    Cycles have $2$-treedepth equal to 3, hence, any graph with $2$-treedepth at most 2 is a forest. 
    The result for $k = 2$ follows immediately, as any path in a forest is induced. 
    Assume that $k \geq 3$ and let $G$ be a graph with $2$-treedepth at most $k$.
    
    Let $P = v_1 \dots v_n$ be a path in $G$. 
    By restricting to the induced subgraph $G[V(P)]$ we may and will assume without loss of generality that $P$ is a Hamiltonian path in $G$. 
    Recall that $\calA(G)$ is the set of cut-vertices of $G$. 
    Let $A \coloneqq \set{i \in [n] : v_i \in \calA(G)}$ and enumerate $A$ as $a_1 < a_2 < \dots < a_r$. Also define $a_0 \coloneqq 1$ and $a_{r + 1} \coloneqq n$ and note that $a_0 < a_1 < a_2 < \dots < a_r < a_{r + 1}$. 
    Since $P$ is a Hamiltonian path in $G$, the blocks of $G$ are $B_i = G[\set{v_{a_i},\dots,v_{a_{i+1}}}]$ for each $i \in \{0,\dots,r\}$.
    %\freddie{Is there a reference to this fact or do we need to add a proof?}.
    %\jedrzej{I'm okay with not proving it. }

    For each $i \in \{0,\dots,r\}$, let $P_i$ be an induced path in $B_i$ between $v_{a_i}$ and $v_{a_{i+1}}$.
    We claim that $Q \coloneqq P_0 v_{a_1} P_1 v_{a_2} \dots v_{a_r} P_r$ is an induced path in $G$. 
    Since the only non-empty intersections between the blocks of $G$ are $V(B_{i - 1}) \cap V(B_i) = \set{v_{a_i}}$, the graph $Q$ is certainly a path. Suppose that there are two adjacent non-consecutive vertices $u$ and $v$ of $Q$. Since each $P_i$ is induced, $u$ and $v$ cannot be in the same $P_i$. However this implies there is no block that contains both $u$ and $v$, which contradicts $uv \in E(G)$.

    The path $Q$ contains $v_{a_0}, \dots, v_{a_{r + 1}}$ and so, if $r \geq n^{1/(k-1)}/2 - 2$, then $Q$ is an induced path with at least $n^{1/(k-1)}/2$ vertices, as required. 
    Otherwise, the number of blocks of $G$ is $r + 1 \leq n^{1/(k-1)}/2 - 1$ and so there is $i \in \{0,\dots,r\}$ such that $B_i$ contains at least $n/(n^{1/(k-1)}/2 - 1) \geq 2 n^{1 - 1/(k-1)} + 1$ vertices.
    By definition of $2$-treedepth, there is a vertex $v \in V(B_i)$ such that $\tdtwo(B_i - v) = \tdtwo(B_i) - 1 \leq \tdtwo(G) - 1 \leq k-1$. Block $B_i$ has the Hamiltonian path $v_{a_i} v_{a_i + 1} \dots v_{a_{i + 1}}$. Either the first or second half of this Hamiltonian path does not contain $v$ and so $B_i - v$ contains a path on at least $(\abs{V(B_i)} - 1)/2 \geq n^{1 - 1/(k-1)}$ vertices. Finally, by induction on $k$, $B_i - v$ (and so $G$) contains an induced path on at least $(n^{1 - 1/(k-1)})^{1/(k - 2)}/2 = n^{1/(k-1)}/2$ vertices, as required.
\end{proof}

\section{Open problems}\label{sec:open}

The first thread of this paper was to determine
\[
    g(k, t) \coloneqq \max\set{\td(G) : \tdtwo(G) \leq k \text{ and } G \text{ is } P_t\text{-free}}.
\]
We like to think of $t$ being a fixed constant and $k$ being a variable.
Then, \cref{thm:main,thm:lower-bound} combine to show that $g(k, t) = \Theta(k^{\lfloor(t-1)\slash 2\rfloor})$.
\Cref{thm:p5free-upper} determines exactly $g(k,4)$ and
$g(k,5)$.
We suspect the following.
\begin{conjecture}\label{conj:poly}
    For all integers $t \geq 2$ and $k \geq 2$, $g(k, t) = \binom{\lfloor (t-1) \slash 2 \rfloor+k-1}{\lfloor (t-1) \slash 2 \rfloor}$.
\end{conjecture}

This would follow if \cref{lem:td-f} held with the recursive definition of $f$ (which appears just before \cref{lem:computation}) altered to remove the `$+1$' from $f(k, t) = f(k, t - 2) + f(k - 1, t) + 1$.

Next, one could similarly define
\[
h(k, t) \coloneqq \max\set{\td(G) : \tw(G) \leq k \text{ and } G \text{ is } P_t\text{-free}}.
\]
From the results mentioned in the paper, we know that for a fixed $t$, the function $h(k,t)$ is in $\Omega(k^{\lfloor (t-1) \slash 2 \rfloor})$ and $\Oh(k^{t-1})$.
\begin{problem}
For a fixed integer $t \geq 2$, determine the growth rate of $h(k, t)$.
\end{problem}
A similar problem can be stated for parameters $\td_3$, $\td_4$, etc.\ as defined by Rambaud~\cite{k-td}.

The second thread of this paper was to understand $f(\calC, n)$, the largest integer such that, for every $G \in \calC$, if $G$ contains $P_n$ as a subgraph, then it contains $P_{f(\calC, n)}$ as an induced subgraph. \Cref{thm:td2-lower,thm:td2-upper,thm:pw} determine $f(\calC, n)$ when $\calC$ is the class of graphs with pathwidth at most $k$ or the class of graphs with 2-treedepth at most $k$. A very interesting open case is when $\calC = \calW_k$ the class of graphs with treewidth at most $k$. Here the current bounds for general $k$ are
\[
\tfrac{1}{4} (\log n)^{1/k} \leq f(\calW_k, n) \leq 2 \floor{k/2} (\log n)^{\frac{1}{\floor{k/2}}},
\]
where the lower bound\footnote{The statement of \cite[Theorem~1.6]{Hilaire2023} is $f(\calW_{k - 1}, n) \geq 1/4 \cdot (\log n)^{1/k}$. However, using the improved bound for $f(\calP_k, n)$ (see footnote~\ref{foot:improved}) in their argument gives the bound written here.} is due to Hilaire and Raymond~\cite{Hilaire2023} and the upper bound is due to Esperet, Lemoine, and Maffray~\cite{ELM17}. For $k = 2$ it is known that $f(\calW_2, n) = \Theta(\log n)$~\cite{AV00,ELM17}. We believe that even the case $k = 3$ is open.
\begin{problem}
    Determine the growth rate of $f(\calW_k, n)$ for each fixed integer $k \geq 3$.
\end{problem}

\section*{Acknowledgements}

We thank Clément Rambaud for pointing out that the weak variant of \cref{thm:main} follows from the known results on weak colouring numbers.

% ----------------------------
\bibliographystyle{abbrv}
\bibliography{bibliography}
\end{document}